\documentclass[12pt]{article}
\usepackage{bbm}
\usepackage{graphicx}
\graphicspath{ {./images/} }
\usepackage{wrapfig}
\usepackage[mathscr]{euscript}
\usepackage{subcaption}
\usepackage{hyperref}
\usepackage{epstopdf}
\usepackage{enumerate}
\usepackage{amsmath,amsfonts,amssymb,amsthm,epsfig,epstopdf,titling,url,array}
\usepackage{color}
\usepackage{framed}
\usepackage[utf8]{inputenc}
\usepackage[english]{babel}

 \usepackage{xparse}
 
 \NewDocumentCommand{\INTERVALINNARDS}{ m m }{
 	#1 {,} #2
 }
 \NewDocumentCommand{\interval}{ s m >{\SplitArgument{1}{,}}m m o }
 {
 	\IfBooleanTF{#1}{
 		\left#2 \INTERVALINNARDS #3 \right#4
 	}{
 		\IfValueTF{#5}{
 			#5{#2} \INTERVALINNARDS #3 #5{#4}
 		}{
 			#2 \INTERVALINNARDS #3 #4
 		}
 	}
 }
 \newtheorem{theorem}{Theorem}[section]
 \newtheorem{corollary}{Corollary}[section] 
 
 \newtheorem{lemma}[theorem]{Lemma}
 
 \newtheorem{example}{Example}[section]
 
 \newtheorem{Conjecture}[theorem]{Conjecture}
 \newtheorem{Remark}{Remark}[section]

\usepackage{authblk}

\begin{document} 
\title{\textbf{Julia sets of rational maps with rotational symmetries}}
\author[1]{Tarakanta Nayak\footnote{tnayak@iitbbs.ac.in}}
\author[1]{Soumen Pal\footnote{Corresponding author: sp58@iitbbs.ac.in}}
\affil[1]{\textit{School of Basic Sciences\hspace{9cm}Indian Institute of Technology Bhubaneswar, India}}
\date{}
\maketitle
\begin{abstract} By a symmetry of the Julia set of a polynomial, also referred as polynomial Julia set, we mean an Euclidean isometry preserving the Julia set. Each such symmetry is in fact a rotation about the centroid of the polynomial.	In this article, a  survey of the symmetries  of polynomial Julia sets  is made. Then  the Euclidean isometries preserving   the Julia set of  rational maps are considered. A rotation preserving the Julia set of a rational map is called a rotational symmetry of its Julia set. A sufficient condition is provided for a rational map to have rotational symmetries whenever the rational map has an exceptional point.
Two classes of rational maps are provided whose Julia sets have rotational symmetries of finite orders. Using this, it is proved that $ z\mapsto \mu z$ where $\mu^{m+n}=1$ is a rotational symmetry of the McMullen map $  z^m+\frac{\lambda}{z^n}$ for all $m,n$  with $m\geq 2$ and $\lambda \in \mathbb{C}\setminus \{0\}$. Assuming that a normalized polynomial has a simple root at the origin, it is shown that the groups of the rotational symmetries of the polynmial coincide with that of its Newton's method and Chebyshev's method.
\end{abstract}
Keywords:
Complex dynamical system; Rational maps; Fatou and Julia sets; Symmetry group of Julia sets.\\
	AMS Subject Classification: 37F10, 65H05
\section{Introduction}
  Complex Dynamics deals with the iteration of analytic functions on $\widehat{\mathbb{C}}=\mathbb{C}\cup \{\infty\}$. For a non-constant rational function $R$, the extended complex plane $\widehat{\mathbb{C}}$ is partitioned into two disjoint sets namely, the Fatou set and the Julia set. The Fatou set of $R$, denoted by $\mathcal{F}(R)$, is defined as the maximal open subset of $\widehat{\mathbb{C}}$ where $\{R^n\}_{n>0}$ is equicontinuous. The complement of $\mathcal{F}(R)$ in $\widehat{\mathbb{C}}$ is called the Julia set of $R$ and is denoted by $\mathcal{J}(R)$. By definition the Fatou set is open, whereas the Julia set is closed. Also note that $\mathcal{F}(R)=\mathcal{F}(R^k)$ for all $k\geq 1$. Further details can be found in ~\cite{Beardon_book,Milnor_book}. 
  \par 
  A point $z_0\in \widehat{\mathbb{C}}$ is said to be a fixed point of $R$ if its image under $R$ is itself. A fixed point can be classified according to its multiplier $\lambda$. If $z_0\in \mathbb{C}$, then $\lambda$ is defined as $\lambda=R'(z_0)$ and whenever $\infty$ is a fixed point of $R$, its multiplier is defined as $g'(0)$ where $g(z)=\frac{1}{R(\frac{1}{z})}$. Now, $z_0$ is said to be an attracting fixed point if $\lambda$ lies in the unit disk (i.e., $|\lambda|<1$). In a particular case, $z_0$ is superattracting whenever $\lambda=0$. If $\lambda$ lies in the exterior of the closed unit disk then $z_0$ is repelling and whenever $\lambda$ is on the unit circle, $z_0$ is said to be indifferent. An indifferent fixed point is said to be rationally indifferent or parabolic if and only if $\lambda$ is a root of unity (i.e., for some $n\in \mathbb{N}$, $\lambda^n=1$), else it is called irrationally indifferent. A point $z^*$ is called a periodic point of $R$ with period $p$, in short a $p$-periodic point if it is a fixed point of $R^p$ but $R^q(z^*)\neq z^*$	for any $q < p$. The classification of $z^*$ can be done in same way considering it to be a fixed point of the rational map $R^p$. The set $\{z^*, R(z^*),\dots, R^{p-1}(z^*)\}$ is called a $p$-periodic cycle. The Fatou set is open but not always connected. A maximal open connected subset of the Fatou set is called a Fatou component. A Fatou component $U$ of of a rational map $R$ is called $p$-periodic if $R^p(U)\subseteq U$. A Fatou component	$U$ is called pre-periodic if it is not periodic but there exist a natural number $k$ such that $R^k(U)$ is periodic. D. Sullivan proved that every Fatou component of a	rational map is either periodic or pre-periodic. In fact, there are four types of Fatou components for a rational map $R$. Let $U$ be a $p$-periodic Fatou component. Then 
  \begin{itemize}
  	\item $U$ is said to be an attracting component if $z_0\in U$, where $z_0$ is an attracting $p$-periodic point.
  	\item $U$ is a parabolic component whenever $\partial U$ contains a parabolic $p$-periodic point.
  	\item $U$ is a Herman ring or a Siegel disk if $R^p:U\mapsto U$ is conformally conjugate to an irrational rotation of some annulus or to the unit disk respectively onto itself.
\end{itemize}
  We say $U$ is a rotation domain if it is either a Herman ring or a Siegel disk.
  The Julia set is completely invariant under the function and is usually fractal with complicated topology. Though the iterative behavior of the function on its Julia set is \textit{chaotic}, it often possesses  some pattern. More precisely, there may exist a M\"{o}bius map $\sigma$ such that $\sigma(\mathcal{J}(R))=\mathcal{J}(R)$. The collection of all such maps, denoted by $\mathcal{M}(R)$, is closed under composition of functions and forms a group. This, or sometimes an appropriate subgroup of it,  gives an idea, at least  approximately about the structure of the Julia set without in fact finding it.
  \par   
 Let $X_i$ be a metric space with the metric $d_i$ for $i=1,2$. A map $h:X_1 \to X_2$ is called  an isometry if   $d_2(h(z) ,h(w) )=d_1(z,w)$ for all $z,w\in X_1$. An isometry is necessarily one-one.  Every analytic Euclidean isometry of $\mathbb{C}$ is of the form  $z\mapsto az+b$ with $|a|=1$.  Such an isometry is either a   translation (if $a=1$) or a rotation about the point $\frac{b}{1-a}$ (if $a\neq 1$).  Similarly, a chordal isometry of $\widehat{\mathbb{C}}$ is a M\"{o}bius map of the form  $z \mapsto \frac{az-\bar{b}}{bz+\bar{a}} $ where $ |a|^2+|b|^2=1$. Here the chordal distance $\rho(z,w)$ between two points in $z,w$ in $\mathbb{C}$ is given by $\frac{2 |z-w|}{\sqrt{1+|z|^2}\sqrt{1+|w|^2}}$ and  $\rho(z,\infty)=\frac{2}{\sqrt{1+|z|^2}}$. Though no translation is a chordal isometry, all rotations about the origin are isometries with respect to the  chordal metric.
 \par
 Let $\mathcal{M}(R)$ be the set of all M\"{o}bius maps preserving the Julia set of $R$.
   We consider two subgroups of $\mathcal{M}(R)$, namely
  $$\mathcal{I}(R)=\{s(z)=\frac{az-\bar{b}}{bz+\bar{a}}: |a|^2+|b|^2=1 \text{ and } s(\mathcal{J}(R))=\mathcal{J}(R)\}$$ and  
  $$\Sigma R=\{\sigma(z)=az+b: |a|=1, a\neq 1 \text{ and }\sigma(\mathcal{J}(R))=\mathcal{J}(R)\}.$$
   A (non-trivial) translation, i.e., map of the form $z \mapsto z+a$ for some $a \neq 0$ cannot be in $\mathcal{I}(R)$. If $\Sigma R$ contains rotations with respect to two distinct points then their composition is a  translation (see the proof of Theorem~\ref{Rotation}) and is in $\Sigma R$.  In other words,  if  $\Sigma R$ does not contain any  translation then each of its elements is a rotation with respect to a point, that depends on $R$ but not on $\Sigma R$. Further, if the point is the origin in this case then  $\Sigma R \subseteq \mathcal{I}(R)$. We call $\Sigma R$ is non-trivial if it contains at least one non-identity element.
\par   
The study of $\Sigma R$, referred as the symmetry group of $R,$ when $R$ is a polynomial is done  by  Julia, Baker, Eremenko and later by Beardon~\cite{BE1987,Beardon1990, Beardon1992,julia-1922}. It is known that, for every polynomial $p$, there is a point $\xi(p)$ such that each element of $\Sigma p$ is a rotation about $\xi(p)$ (see Lemma \ref{Rotation}). The study of symmetry in rational maps remains relatively underexplored. While some literature exists, such as references \cite{Levin1991,LP1997,Ye2015}, discussing rational maps having identical Julia sets, the subject still lacks extensive study. Recently, Ferreira made a systematic study of $\mathcal{I}(R)$ for all rational maps $R$ (see \cite{Ferreira2019}). Each element of $\mathcal{I}(R) $ is a rotation of the sphere with respect to some axis passing through the origin. But a rotation in the plane with respect to a non-zero point is not in $\mathcal{I}(R)$. Also, this set $\mathcal{I}(R)$ does not contain any translation. Thus, Ferreira's work does not accommodate Julia sets that are preserved under two geometrically simple classes of maps, namely translations and rotations of the plane with respect to a non-zero point. The later maps are called rotational symmetries.   This article presents a survey of results on rotational symmetries of polynomial Julia sets and the related issue of identical Julia sets for  two different polynomials. 
\par  
Let $R$ be a rational map  analytic at $z_0 \in \mathbb{C}$. If its Taylor series about $z_0$ is given by $a_k (z-z_0)^k +a_{k+1} (z-z_0)^{k+1}+\cdots $ for some $k > 0$ where $a_k \neq 0$ then we say the local degree of $R$ at $z_0$, denoted by $\deg(R,z_0)$ is $k$. The map $R$ is \textit{like} $z \mapsto z^k$ near $z_0$. The local degree of $R$ at $\infty$ or at a pole is defined by a change of coordinate using $z \mapsto \frac{1}{z}$. More precisely, if $R(\infty) \in \mathbb{C}$ then $\deg(R,\infty)$ is defined as the degree of $R(\frac{1}{z})$ at $0$. If $R(\infty)=\infty$ then $\deg(R,\infty)$ is  defined as $\deg(\frac{1}{R(\frac{1}{z})},0)$. A point $w$ is exceptional  for a rational map $R$ if $\deg(R,w)$ is equal to the degree of $R$. This is equivalent to the statement that  $\{z: R^n(z)=w~\mbox{for some}~ n \geq 0\}$ is finite. There can be at most two exceptional points for any rational map (Theorem 4.1.2., \cite{Beardon_book}). This article shows, under some condition that a rational map with an exceptional point has rotational symmetry. Two classes of rational maps are presented whose Julia sets have rotational symmetries. 

  \par 
  In Section 2, a systematic discussion of the symmetry group of polynomial Julia sets is made. Results relating the rotational symmetries of polynomial Julia sets with two polynomials with identical Julia sets are dealt with in Section \ref{Results on poly}.  Section 4 deals with the rotational symmetries of rational  Julia sets.   If a rational map $R$  has an exceptional point then existence of rotational symmetries  of the Julia sets of $R$ is proved under some condition (see Theorem \ref{Conj to poly}). We also introduce two forms of rational maps whose Julia sets  have rotational symmteries of finite order (see Theorem \ref{Form1}, \ref{Form2}).
   \par All the polynomials and rational maps are assumed to be of degree at least two, unless stated otherwise. By a translation or a rotation, we mean a non-identity translation or rotaion respectively, unless stated otherwise. 
 \section{Symmetries of polynomial Julia sets}
Let
 \begin{equation}\label{poly}
 p(z)=a_dz^d+a_{d-1}z^{d-1}+\dots +a_0
 \end{equation} where $a_d\neq 0$ and $d\geq 2$. The centroid of $p$, denoted by $\xi$ is defined as $\xi =-\frac{a_{d-1}}{da_{d}}$. For $c\in \mathbb{C}$, the equation $p(z)=c$ has $d$ number of roots counting with multiplicities. A root $z^*$ is counted $m$-times here if it is with multiplicity $m$ i.e., $p(z)-c=(z-z^*)^m h(z)$ for some analytic function $h$ in a neighborhood of $z^*$ such that $h(z^*)\neq 0$. If the roots  of $p(z)=c$ are $z_1,z_2,\dots,z_d$, repeated according to their multiplicities, then $p$ can be expressed as $p(z)=c+a_d \prod_{i=1}^{d}(z-z_i)$. Comparing the coefficients of $z^{d-1}$ on both the sides, we get $\sum_{i=1}^{d}z_i=-\frac{a_{d-1}}{a_{d}}.$ Therefore, $\xi$ is the average of the roots of $p(z)=c$ i.e., $\frac{1}{d}\sum_{i=1}^{d}z_i=-\frac{a_{d-1}}{da_{d}}$. It is important to note that $\xi$ is independent of $c$. Observe that $\xi =0$ if and only if $a_{d-1}=0.$ A polynomial whose centroid is at the origin, is called centered. 
 \par
 For any polynomial $p$ as given in (\ref{poly}), consider the affine map $\psi(z)=Az+\xi$, where $\xi$ is the centroid of $p$ and $A$ is such that $A^{d-1}=\frac{1}{a_d}$. Then the polynomial $g=\psi^{-1}\circ p\circ \psi$ is monic and centered. Such a polynomial is called normalized. The fact that every polynomial is conjugate to a normalized polynomial is crucial for investigating the symmetries of polynomial Julia sets. In fact, we have the following for all rational maps (see Theorem 3.1.4., \cite{Beardon_book}).
 \begin{lemma}\label{conj-JS}
 	  If $R_1$ and $R_2$ are two rational maps such that $R_1=\psi^{-1}\circ R_2 \circ \psi$ for some M\"{o}bius map $\psi$ then  $\mathcal{J}(R_2)=\psi(\mathcal{J}(R_1))$. 
 \end{lemma}
 For every   $\sigma \in \Sigma p$, $\psi^{-1} \circ \sigma \circ \psi$  is an Euclidean isometry preserving the Julia set of $g$. Conversely, if $\gamma \in \Sigma g$ then $\psi \circ \gamma \circ \psi^{-1}$ is an Euclidean isometry preserving the Julia set of $p$. Thus, we have the following.
  \begin{lemma}\label{Prop1}
 For every polynomial $p$, there is an affine map $T$ such that  $g=\psi ^{-1}\circ p\circ \psi$ is normalized and  $\Sigma p=\psi \circ (\Sigma g)\circ \psi^{-1}$.
 \end{lemma}
Note that the affine conjugacy $\psi$ maps $0$, the centroid of $g$ to that of $p$.
 Now onwards, we discuss the symmetry group of normalized polynomials and without loss of generality assume that $p$ is normalized.	
 If $p$ does not have any constant term in its expression, i.e., if $p(0) =0$ then  take $z^\alpha$ common from the expression of $p$ where $\alpha$  is the multiplicity  of $0$ as a root of $p$. 
 \par If $p$ has a non-zero constant term, we take $\alpha=0.$ Hence it is always possible to express $p$ as $p(z)=z^\alpha p_1(z)$, where $p_1$ is a normalized polynomial, $p_1(0)\neq 0$, and $\alpha\in \mathbb{N}\cup \{0\}$ is maximal for this form. Let $\beta_1, \beta_2,\dots,\beta_k$ be the (non-zero) powers of $z$ in the expression of $p_1$ and $\beta=\gcd (\beta_1, \beta_2,\dots,\beta_k)$. Then $p_1$ can be expressed as $p_1(z)=z^{m_1\beta}+a_2z^{m_2\beta}+\dots +a_kz^{m_k\beta}+a_{k+1},$ where $a_i\neq 0$ for $i=2,3,\dots,k+1$, $m_j\in \mathbb{N}$ for $j=1,2,\dots, k$, and $\gcd (m_1, m_2,\dots m_k)=1.$ Hence,
 \begin{equation}\label{norm}
 p(z)=z^\alpha p_0(z^\beta)
 \end{equation}
 where 
\begin{equation}\label{form p_0}
p_0(z)=z^{m_1}+a_2z^{m_2}+\dots +a_kz^{m_k}+a_{k+1}
\end{equation} 
 is a monic polynomial. Note that $p_0$ is not necessarily centered whereas  $p_1$ is always so.
 Further note that, $\alpha$ and $\beta$ are maximal for the expression (\ref{norm}) and they determine $p_0$ completely.
 \par 
  If $\lambda_1, \lambda_2,\dots, \lambda_r$ are the distinct roots of $p_0$ with multiplicities $b_1,b_2,\dots, b_r$ respectively, then we can write $p$ as $$p(z)=z^\alpha\prod_{s=1}^{r}(z^\beta -\lambda_s)^{b_s}.$$ Therefore, all the non-zero roots of $p$ can be partitioned into $r$ number of sets $A_s=\{z:z^\beta=\lambda_s\}$ where $s=1,2,\dots,r.$ Each element of $A_s$ lies on a circle of radius $|\lambda_s|^{1/\beta}$ around the origin, and each of these differs from its nearest one by an argument of $\frac{2 \pi}{\beta}$. Hence each rotation about the origin of order $\beta$ preserves every $A_s$. In other words, each such rotation takes a root of $p$ to another root with the same modulus and with the same multiplicity. These rotations are going to be the elements of $\Sigma p$.
 \par

 Since $\infty$ is a superattracting fixed point of each polynomial, it  has a neighborhood  contained in the Fatou set. Thus we have the following.
\begin{lemma}\cite{Beardon_book}
	For every polynomial $p$, $\mathcal{J}(p)$ is bounded.
	\label{bded}
\end{lemma}
An Euclidean isometry $\sigma(z)=az+b, |a|=1$ is either a translation (if $a=1$) or a rotation about the point $\frac{b}{1-a}$ (if $a\neq 1$). Indeed, $\phi \sigma \phi^{-1}(z)=az$ where $\phi(z)=z-\frac{b}{1-a}.$ Now we identify possible elements of $\Sigma p$.
\begin{lemma}\label{Rotation}
	For each polynomial $p$, there is a point $\xi(p)$  such that every element of $\Sigma p$ is a rotation about $\xi(p)$.
\end{lemma}
\begin{proof}
 If there is a translation $T$ in $\Sigma p$ then $T^n\in \Sigma p$ for all $n\in \mathbb{N}$. For $z\in \mathcal{J}(p)$, $T^n (z) \in \mathcal{J}(p)$ whereas  $T^n(z)\rightarrow \infty$ as $n\rightarrow \infty$. This gives that $\mathcal{J}(p)$ is unbounded, contradicting Lemma~\ref{bded}. Therefore $\Sigma p $ does not contain any non-trivial (i.e., non-identity) translation.
	\par 
	Suppose that $\sigma, \gamma \in \Sigma p$ are two non-trivial (i.e., non-identity) rotations about two distinct points $\alpha$ and $\beta$ respectively. Then $\sigma(z)=ze^{i\theta}+\alpha(1-e^{i\theta})$ and $\gamma(z)=ze^{it}+\beta(1-e^{it})$, for some $\theta, t \in (0,2\pi).$  Note that $\sigma^{-1}(z)=ze^{-i\theta}+\alpha(1-e^{-i\theta})$ and $\gamma^{-1}(z)=ze^{-it}+\beta(1-e^{-it})$. As $\Sigma p$ is a group under composition of functions, $\gamma\circ \sigma \circ\gamma^{-1}\circ \sigma^{-1}\in \Sigma p.$ But $\gamma\circ \sigma \circ\gamma^{-1}\circ \sigma^{-1}(z)=z+(\alpha-\beta)(e^{it}-1)(1-e^{i \theta})$,
	which is a non-trivial translation. Since this is already known to be impossible, $\alpha=\beta$. Therefore, there is a point $\xi(p)$ such that every element of $\Sigma p$ is a rotation about $\xi(p)$.
\end{proof}
 Theorem \ref{Rotation} states that all the elements of $\Sigma p$ are  rotations  with respect to a single point, which possibly  depends on $p$. What can that point be? To answer this question, recall that the average of all solutions of the equation $p(z)=c$ is $\xi$, the centroid of $p$, irrespective of the value of $c$. For every root $w$ of $p(z)=c$, the average of all solutions of $p(z)=w$ is also $\xi$. In general, the average of all solutions of $p^{-n}(z)=c$ is $\xi$, for every $n\in \mathbb{N}$ and $c\in \mathbb{C}.$ Now consider a  point $z_0\in \mathcal{J}(p)$. By the backward invariance of the Julia set, every open set containing the Julia set contains the set $\{z: p^n(z)=z_0\}$ for all $n$. In fact, the set $\{z: p^n(z)=z_0\}$ is in a sense \textit{uniformly distributed} in the Julia set. To see it, let $\epsilon>0$ and the Julia set of $p$ be covered by finitely many balls $B_i, i=1,\cdots k$, each with radius $\frac{\epsilon}{2}$. This is possible as $\mathcal{J}(p)$ is compact.  Since $z_0$ is not exceptional (because all exceptional points belong to the Fatou set), there is an $n_i$ such that $z_0 \in p^{n}(B_i)$ for all $n >n_i$ (by Theorem 6.9.4,~\cite{Beardon_book}). If $N =\max\limits_{1 \leq i \leq k} \{n_i\}$ then $z_0 \in p^{n}(B_i)$ for all $n >N$ and for all $i$. Let $z_i \in B_i$ such that $p^{n}(z_i)=z_0$. Here $z_i$ depends on $n$. Now the union of balls with radius $\epsilon$ and with center at the points of $\{z: p^n(z)=z_0\}$ contains $\mathcal{J}(p)$ for all $n >N$.   This is  a reason why the centroid is expected to be the point stated in Theorem \ref{Rotation}.
 
 If a normalized polynomial $p$ is  affine conjugate to a monomial then its Julia set is a circle whose center is $0$, the centroid of $p$. In this case, $\Sigma p$ contains all the rotations about $0$. Therefore, $\Sigma p$ is an infinite set. Beardon proves that the converse  of this statement is also true (see Lemma 4, \cite{Beardon1990}). 
 
 Consider $p$ which is not conjugate to any monomial. There is a conformal map $\phi$ in a neighborhood of $\infty$, called the B\"{o}ttcher coordinate into the unit disk such that $\phi \circ p \circ \phi^{-1}(z)=z^d$ where $d$ is the degree of $p$. The function $|\phi|$ extends continuously to the whole basin of attraction, $\mathcal{A}$ of $\infty$. Using this $\phi$, the Green's function $\log |\phi(z)|$ is defined in   $\mathcal{A}$ with the pole at $\infty$ and further analysis gives that every element of $\Sigma p$ is rotation about the origin. In fact, Beardon proved the following.
 
 \begin{theorem}[\cite{Beardon_book}]
 	Let $p$ be a normalized polynomial of degree $d\geq 2$. A rotation $\sigma$ of finite order about the origin is in $\Sigma g$ if and only if $p \circ \sigma=\sigma^d\circ p.$
 \end{theorem}

The symmetry group of Julia set of a normalized polynomial is now described.
 \begin{theorem}(\cite{Beardon_book})
 	If $p$ is a normalized polynomial of the form (\ref{norm}) then
 	 $\Sigma p=\{\sigma: \sigma(z)=\lambda z, \lambda^\beta=1\}.$ 
 	\label{poly-symm}
 \end{theorem}
\begin{Remark}
It is easy to observe that the above theorem is true even if $p$ is not monic but only centered.	
	\end{Remark}

 In view of Lemma~\ref{Prop1}, the symmetry group of the Julia set of an arbitrary polynomial $q$ with centroid at $\xi$ that is conjugate to a normalized polynomial $p$, i.e., $p = \psi^{-1} \circ q \circ \psi$ for $\psi(z)=Az +\xi$ for a suitable $A$ (see Lemma~\ref{Prop1}) is given by 
$$\Sigma q=\{\sigma: \sigma(z)=\lambda (z-\xi)+\xi, \lambda^\beta =1\},$$ where $p$ is as given in Theorem~\ref{poly-symm}.

We now discuss few examples.
\begin{example}
\begin{enumerate}
	\item The polynomial $q(z)=z^3+3z^2+3z-\frac{1}{3} $   is not  normalized  and its centroid is $\xi=-1.$ For the affine map $\psi(z)=z-1$, the polynomial $\psi^{-1}\circ q\circ \psi(z)=z^3-\frac{1}{3}$ is clearly  normalized. Writing it in the form (\ref{norm}), it is observed that $\alpha=0$ and $\beta=3$. Therefore, $\Sigma (\psi^{-1} \circ q \circ \psi)=\{z\mapsto \lambda z: \lambda^3=1\}$. Hence, $\Sigma q=\psi \circ(\Sigma p)\circ \psi^{-1}=\{z\mapsto \lambda(z+1)-1: \lambda^3=1\}$ (see Fig. \ref{Graphs}(a)).
	\item The polynomial $p(z)=z^3-1.2i z$  is   normalized and is in the prescribed form (\ref{norm}) with $\alpha=1$ and $ \beta=2.$ Therefore, $\Sigma p=\{z\mapsto \lambda z: \lambda^2=1\}$ (see Fig. \ref{Graphs}(b)).
	\item The symmetry group of the normalized polynomial $p(z)=z^3-z-0.5i$ is trivial as it is in the form (\ref{norm}) where $\alpha=0$ and $\beta=1.$ However, there is a non-M\"{o}bius homeomorphism preserving its Julia set. In fact, $p(-\bar{z})=(-\bar{z})^3-(-\bar{z})-0.5i=-(\bar{z^3}-\bar{z}-\overline{0.5i})=-\overline{p(z)}$. Therefore, $\mathcal{J}(p)$ is preserved under the reflection about the imaginary axis (see Fig. \ref{Graphs}(c)). That $z \mapsto -\overline{z}$ is a chordal isometry is used here.
\end{enumerate}
\begin{figure}[h!]\label{q}
	\begin{subfigure}{.5\textwidth}
		\centering
		\includegraphics[width=0.98\linewidth]{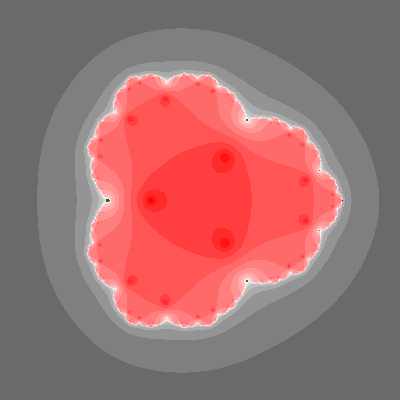}
		\caption{Julia set of $ z^3-\frac{1}{3} $}
	\end{subfigure}%
	\begin{subfigure}{.5\textwidth}
		\centering
		\includegraphics[width=0.98\linewidth]{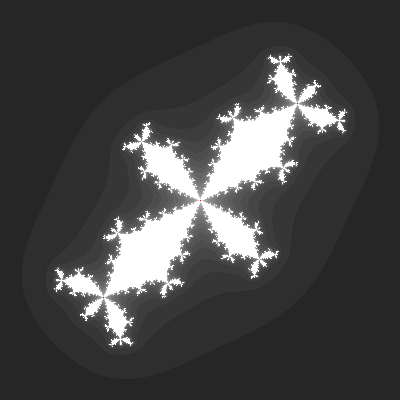}
		\caption{Julia set of $ z^3-1.2i z $}
	\end{subfigure}\\[1ex]
	\centering
	\begin{subfigure}{0.5\textwidth}
		\centering
		\includegraphics[width=1\linewidth]{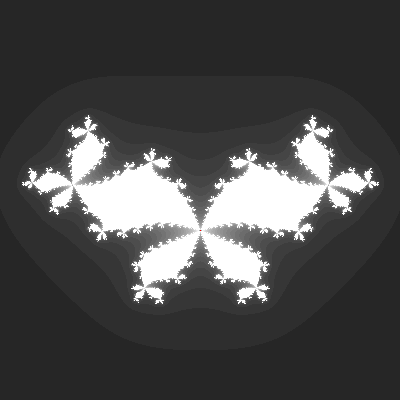}
		\caption{Julia set of $ z^3-z-0.5i $}
	\end{subfigure}
	\caption{Symmetries of the Julia sets }
	\label{Graphs}
\end{figure}
\end{example}
%
\section{Polynomials with the same Julia set}\label{Results on poly}
  \textit{When two polynomials have the same Julia set?} This question is closely related to the symmetries of the Julia set. We start with the following result which  is  proved by Julia in 1922. 
 \begin{theorem}\cite{julia-1922}
 If two polynomials $p$ and $q$ commute (i.e., $p\circ q=q\circ p$) then $\mathcal{J}(p)=\mathcal{J}(q)$.
\label{julia-1922}
 \end{theorem}
The above result is also true for all rational maps and a proof can be found in Theorem 4.2.9.,~\cite{Beardon_book}. A kind of converse is obtained by Baker and Eremenko.
\begin{theorem}\cite{BE1987}
If two polynomials $p$ and $q$ have the same Julia set $\mathcal{J}$ then either the polynomials commute or there exists a non-identity Euclidean isometry $\sigma$ such that $\sigma(\mathcal{J})=\mathcal{J}$.  
\end{theorem}
After Ritt \cite{Ritt1920,Ritt1923} initiated the study on commuting rational maps, numerous subsequent studies have been conducted. Examples include references \cite{Masuda2008,Pakovich2021,Schmid1992}. If $\Sigma p$ consists of the identity only then  $\mathcal{J}(p)= \mathcal{J}(q)$ guarantees that the polynomials commute. But if the symmetry group of $p$ contains at least one non-identity element, the converse of Theorem~\ref{julia-1922}  may not be true. To see it, consider $p(z)=z^2-1$ and $q(z)=-z^2+1$. Then $q(z)=-p(-z)$ and this gives that $\mathcal{J}(q)=\sigma(\mathcal{J}(p) )$ where $\sigma(z)=-z$. Also,  $\sigma \in  \Sigma p$ by Theorem~\ref{poly-symm}. Therefore,  $\mathcal{J}(q)=\sigma(\mathcal{J}(p))=\mathcal{J}(p)$ (see Fig. \ref{same J}). However $p(q(z))=z^4-2z^2=-q(p(z))$.
	\begin{figure}[h!]
	\centering
	\includegraphics[width=0.4\linewidth]{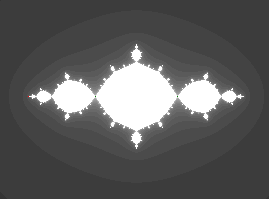}
	\caption{The Julia set of $z^2-1$ and $-z^2+1$}
	\label{same J}
\end{figure}
 More generally, we obtain the following result.
\begin{theorem}
	Let $p$ be a normalized polynomial of the form (\ref{norm}) such that $\alpha \neq 1$ and $\beta \geq 2$. If $\beta$ does not divide $(\alpha-1)^2$ then for every non-identity $\sigma \in \Sigma p$, $q :=\sigma \circ p \circ \sigma^{-1}$ has the same Julia set as that of $p$ but $q \circ p \neq p \circ q$.   
\end{theorem}
\begin{proof}
 For each $\sigma\in \Sigma p$, $\mathcal{J}(p)=\sigma(\mathcal{J}(p))$. Also, as $q =\sigma \circ p \circ \sigma^{-1}$, $\mathcal{J}(q)=\sigma (\mathcal{J}(p))$ (by Theorem 3.1.4., \cite{Beardon_book}). Hence we get $\mathcal{J}(p)=\mathcal{J}(q).$
  
  Any non-identity $\sigma\in \Sigma p$ is of the form $\sigma(z)=\lambda z$, where $\lambda^\beta=1$ (by Theorem~\ref{poly-symm}). Thus,
  \begin{align*}
  q(z)&=\sigma p\left(\frac{1}{\lambda}z\right)=\sigma\left[\frac{1}{\lambda^\alpha}z^\alpha p_0(z^\beta)\right]\\
  &=\frac{1}{\lambda^{\alpha-1}}z^\alpha p_0(z^\beta).
  \end{align*}
  Therefore, 
  \begin{align*}
  p(q(z))&=p\left( \frac{1}{\lambda^{\alpha-1}}z^\alpha p_0(z^\beta)\right)=\left[\frac{1}{\lambda^{\alpha-1}}z^\alpha p_0(z^\beta)\right]^\alpha p_0\left(\left[\frac{1}{\lambda^{\alpha-1}}z^\alpha p_0(z^\beta)\right]^\beta\right)\\
  &=\frac{1}{\lambda^{\alpha(\alpha-1)}}z^{\alpha^2}[p_0(z^\beta)]^\alpha p_0\left (z^{\alpha\beta}(p_0(z^\beta))^\beta\right).
  \end{align*}
 The last equation is due to the fact that $\left(\frac{1}{\lambda^{\alpha-1}}\right)^\beta=\left(\frac{1}{\lambda^\beta}\right)^{\alpha-1}=1.$ However,
  \begin{align*}
  q(p(z))&=q(z^\alpha p_0(z^\beta))=\frac{1}{\lambda^{\alpha-1}}z^{\alpha^2}[p_0(z^\beta)]^\alpha p_0\left(z^{\alpha\beta}(p_0(z^\beta))^\beta\right).
  \end{align*}
  This implies that the polynomials $p$ and $q$ commute if and only if $ \frac{1}{\lambda^{\alpha(\alpha-1)}}=\frac{1}{\lambda^{\alpha-1}} $, which gives $\lambda^{(\alpha-1)^2}=1.$ As $\lambda =e^{\frac{2 \pi i}{\beta}}, \alpha \neq 1$ and $\beta >0$, $k>0$ and consequently, $\beta$ divides $(\alpha -1)^2$. Therefore, if  $\beta$ does not divide $(\alpha-1)^2  $ then for every non-identity $\sigma \in \Sigma p$, $\mathcal{J}(\sigma \circ p \circ \sigma^{-1}) =\mathcal{J}(p)$ but $p \circ q \neq q \circ p$. Since $\beta \geq 2$, there is a non-identity element in $\Sigma p$, and we are done.
\end{proof}
In 1990, Beardon established a necessary and sufficient condition for polynomials with the same Julia set. In a way, this is a complete description of the relation between the polynomials with identical Julia sets and the rotational symmetry of the Julia set.
\begin{theorem}(\cite{Beardon1990}) The polynomials $p$ and $q$ share the same Julia set if and only if there is some $\sigma \in \Sigma p$ such that $p\circ q=\sigma q\circ p.$
\end{theorem} 

The previous four theorems are about the uniqueness of the Julia set of different polynomials. Now the uniqueness of polynomials is considered assuming some relation between their Julia sets.  In this line, there is a result by Fern\'{a}ndez.
\begin{theorem}(\cite{Fer1989})
	 Let $p$ and $q$ be polynomials of the same degree and with the same leading coefficient. If the Julia set of $p$ is disjoint from the unbounded Fatou component (i.e., the basin of $\infty$) of $q$ then $p=q.$
\end{theorem}
As a consequence, it is obtained that if two polynomials with the same degree and the same leading coefficient have the same Julia set then they are the same.
\par

Beardon revealed a beautiful connection between the polynomials with the same degree having identical Julia set.
\begin{theorem}(\cite{Beardon1992})
 Let $p$ be a polynomial of degree $d$. Any polynomial $q$ of the same degree $d$ has the same Julia set as that of $p$ if and only if $q=\sigma p$ for some $\sigma\in \Sigma p.$	 
\end{theorem}
If $\Sigma p$ is trivial then the above theorem gives that if $q$ is a polynomial with the same degree and with the same Julia set as that of $p$ then $q=p$.

We conclude with a result with the same spirit by Schmidt and Steinmetz~\cite{SS1995}.
\begin{theorem}
	 Let $\mathcal{J}$ be a Julia set of a polynomial  which is neither a circle nor a line segment. Then there exists a polynomial $p$ such that any polynomial $q$ with the Julia set $\mathcal{J}$ can be written in the form $q(z)=\sigma p^n(z)$, where $\sigma$ is a rotation (including identity) with $\sigma(\mathcal{J})=\mathcal{J}$, and $n$ is a natural number. 
\end{theorem}
 
\section{Symmetries of rational Julia sets}
This section deals with symmetries of rational maps that are not polynomials. Recall that   $\mathcal{M}(R)=\{\phi:~\phi~\mbox{is a M\"{o}bius map such that }~ \phi(\mathcal{J}(R))=\mathcal{J}(R)\}.$ The set $\mathcal{M}(R)$ is respected by conformal conjugacy in the same way as the affine conjugacies respect the set of Euclidean isometries of polynomial Julia sets (see Lemma \ref{Prop1}).
\begin{lemma}\label{Prop2}
	If two rational maps $R$ and $S$ are conjugate with the M\"{o}bius map $\phi$, such that $S=\phi\circ R\circ \phi^{-1}$, then $\mathcal{M}(S)=\phi\circ (\mathcal{M}(R))\circ \phi^{-1}$.
\end{lemma}

\subsection{Rational maps  with an exceptional point}

Polynomials are rational maps with at least one exceptional point. It is well-known that a rational map has at most two exceptional points. If it has exactly two exceptional points then it is conjugate to $z^d$ for some non-zero integer $d$ (Theorem 4.1.2, \cite{Beardon_book}). To see it, 	Let $\zeta_1$ and $\zeta_2$ be the two exceptional points of $R$. Then, there exists a complex number $a\in \mathbb{C}$ such that for the M\"{o}bius map $\phi(z)=a\frac{z-\zeta_1}{z-\zeta_2}$, the function $p(z)=\phi \circ R \circ \phi^{-1}(z)$ is $z^d$ for some non-zero integer $d$. In both the cases, $\mathcal{J}(p)$ is the unit circle. Further, $\mathcal{J}(R)=\phi^{-1}(\mathcal{J}(p))$ by Lemma~\ref{conj-JS}. Here $\phi^{-1}(z)=\frac{\zeta_2 z-a\zeta_1}{z-a}$ and has its pole at $a$. Hence the Julia set of $R$ is either a circle (if $|a|\neq 1$) or a straight line ($|a|=1$). This is so because the image of the unit circle under any M\"{o}bius map is either a circle or a straight line, and it is a straight line if and only if the pole of the map is on the unit circle.
%
\par 
If a rational map $R$ has exactly one exceptional point, say $w$ then for $h(z)=z-w$, $h \circ R \circ h^{-1}$ is a rational map whose only exceptional point is $0$.   A rotation  about $w$ preserves $\mathcal{J}(R)$ if and only if a rotaion by the same angle about the origin preserves $\mathcal{J}(h \circ R \circ h^{-1})$. Since the rotational symmetries are the main concern of this article,  we assume without loss of generality that the exceptional point of $R$ is $0$. Since $0$ is the only exceptional point of $R$, $R^{-1}(0)=\{0\}$ and therefore, $R$ is of the form
\begin{equation}\label{exceptional-form}
R(z)=\frac{z^d}{a_0z^d+a_1z^{d-1}+\dots+a_{d-1}z+a_d},
\end{equation}
where $a_d\neq 0$.
 As $R$ has exactly one exceptional point it  is conjugate to a (non-monomial) polynomial and hence it is conjugate to a normalized polynomial. Two different  polynomials that are conjugate to $R$ are conjugate to the same normalized polynomial. Let  $\beta$ be the order of the  symmetry group of this normalized polynomial. This $\beta$ is a property of $R$, and we call it the order of rotational symmetries of $R$.  
  The following theorem finds all M\"{o}bius maps, arising out of this normalized polynomial and preserving the Julia sets of $R$.
 
\begin{theorem}\label{Conj to poly}
	If $R$ is rational map with only one exceptional point, that is $0$  and is of the form (\ref{exceptional-form}) then  $ \mathcal{M}(R) $ contains the set $\{z\mapsto \frac{z}{\zeta(1-\lambda)z+\lambda}:\lambda^\beta=1\}$, where $\zeta=-\frac{a_{d-1}}{da_d}$ and $\beta$ is the order of rotational symmetries of $R$.
\end{theorem}
\begin{proof}
	For any M\"{o}bius map $\varphi$, if $\varphi \circ R\circ \varphi^{-1}$ is a polynomial, then $\varphi(0)=\infty$ and  $\varphi(z)=\frac{Az+B}{z}$ for some  $A, B \in \mathbb{C} $ and $ B \neq 0$. Let $p_1(z)=\varphi\circ R\circ \varphi^{-1}(z)$. Then
	\begin{align*}
	p_1(z)&=\varphi\circ R\left(\frac{B}{z-A}\right)\\
	&=\varphi\left(\frac{B^d}{a_d(z-A)^d+Ba_{d-1}(z-A)^{d-1}+\dots+a_0B^d}\right)\\
	&=\frac{(a_d(z-A)^d+Ba_{d-1}(z-A)^{d-1}+\dots+a_0B^d)+AB^{d-1}}{B^{d-1}}.
	\end{align*}
	Here  the centroid of $p_1$ is $\xi=-\frac{-da_dA+Ba_{d-1}}{da_d}=A+B\zeta$ where $ \zeta$ is as given in the statement of this theorem. Now consider the affine map $\psi(z)=\alpha z+\xi$, where $\alpha^{d-1}=  \frac{B^{d-1}}{a_d}$. Then $p(z)=\psi^{-1}\circ p_1\circ\psi (z)$ is a normalized polynomial. Let $p(z)=z^\alpha p_0(z^\beta)$ where $\alpha$ and $\beta$ are maximal for this expression. Then $p(z)=T\circ R\circ T^{-1}(z)$, where $T(z)=\psi^{-1}\circ \varphi(z)=\psi^{-1}\left(\frac{Az+B}{z}\right)=\frac{(A-\xi)z+B}{\alpha z}$.
	\begin{figure}[h!]
		\centering
		\includegraphics[width=0.5\linewidth]{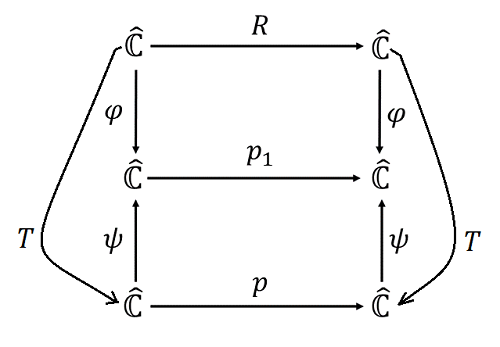}
		\caption{The conjugacy}
	\end{figure}
If $\sigma\in \Sigma p$, where $\sigma(z)=\lambda z$ then  $\lambda^\beta=1$, $T^{-1}\circ\sigma\circ T\in \mathcal{M}(R)$. Now,
	\begin{align*}
	T^{-1}\circ\sigma\circ T(z)&=T^{-1}\circ\sigma\left(\frac{(A-\xi)z+B}{\alpha z}\right)=T^{-1}\left(\lambda \frac{(A-\xi)z+B}{\alpha z}\right)\\
	&=\frac{B}{\alpha \lambda\left(\frac{(A-\xi)z+B}{\alpha z}\right)-A+\xi}=\frac{Bz}{(1-\lambda)(\xi-A)z+\lambda B}\\
	&=\frac{z}{(1-\lambda)\zeta z+\lambda}.
	\end{align*}
	This gives that $\{z\mapsto \frac{z}{\zeta(1-\lambda)z+\lambda}:\lambda^\beta=1\}\subseteq\mathcal{M}(R)$.
\end{proof}
Here are few remarks.
\begin{Remark}
	\begin{enumerate}
	\item Observe that $\zeta$ is the centroid of the polynomial $\frac{1}{R(\frac{1}{z})}$.
	\item  The set $\{z\mapsto \frac{z}{\zeta(1-\lambda)z+\lambda}:\lambda^\beta=1\}$ is independent of the choice of the M\"{o}bius maps that conjugate $R$ to a normalized polynomial.
	\end{enumerate}
\end{Remark}
  This theorem guarantees that, if $R$ is conjugate to a polynomial whose Julia set is invariant under a rotation, then $\mathcal{M}(R)$ contains a non-identity M\"{o}bius map. In a particular, we get the following.
\begin{corollary}\label{Equality_R}
	If $a_{d-1}=0$ in the Equation~\ref{exceptional-form} then $\{z\mapsto\lambda z:\lambda^\beta=1\}\subseteq\mathcal{M}(R)$.
\end{corollary}
Further, if we consider $\Sigma R$ then applying few more restrictions we get the equality in the relation between $\Sigma p$ and $\Sigma R$.
\begin{corollary}
	If $a_{d-1}=0$ in the Equation~\ref{exceptional-form}, $\beta\geq 2$ and $\mathcal{J}(R)$ is not translation invariant then $\{z\mapsto\lambda z:\lambda^\beta=1\}=\Sigma R$.
\end{corollary}
\begin{proof}
	Let $p(z)=\frac{1}{R(\frac{1}{z})}$. The relation $\Sigma p \subseteq \Sigma R$ follows from Corollary \ref{Equality_R}. As $\mathcal{J}(R)$ is not translation invariant, $\Sigma R$ contains rotations about origin. Let there is a rotation $ \sigma(z)=\lambda z $, $|\lambda|=1$ of order $\beta_1>\beta$. Now $\phi \circ \sigma \circ \phi^{-1}(z)=\frac{z}{\lambda}$, where $\phi(z)=\frac{1}{z}$, is a rotation of order $\beta_1$ and as $\mathcal{J}(p)=\phi(\mathcal{J}(R))$, we get $\phi \circ \sigma \circ \phi^{-1}\in \Sigma p$. But it contradicts the maximality of $\beta$ in the expression of $p$. Hence, $\{z\mapsto\lambda z:\lambda^\beta=1\}=\Sigma R$.
\end{proof}
Here are two examples of rational maps $\frac{3z^3}{3-z^3}$ and $\frac{z^3}{1-2z^2}$ that are conjugate to $z^3 -\frac{1}{3} $ and $z^3 -2z$ respectively via the M\"{o}bius map $z\mapsto \frac{1}{z}$. For both the cases $\beta=3$, giving that the Julia sets of the mentioned rational maps are not translation invariant. Their Julia sets are provided in Fig.~\ref{JS-form2} as the boundary of different colors (of blue and yellow in (a) and of red and green in (b)).

%
%
%
%
%
\begin{figure}[h!]
	\begin{subfigure}{.5\textwidth}
		\centering
		\includegraphics[width=0.98\linewidth]{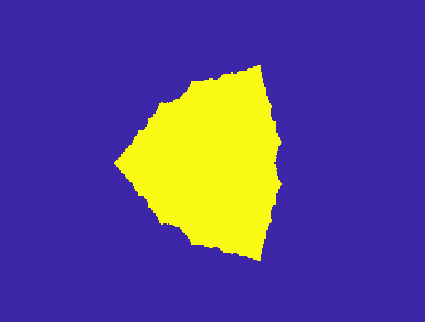}
		\caption{Julia set of $R(z)=\frac{3z^3}{3-z^3}\sim z^3-\frac{1}{3} $}
	\end{subfigure}%
	\begin{subfigure}{.5\textwidth}
		\centering
		\includegraphics[width=0.98\linewidth]{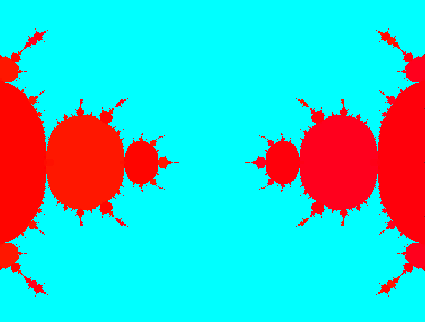}
		\caption{Julia set of $R(z)=\frac{z^3}{1-2z^2}\sim z^3-2z $}
	\end{subfigure}
\caption{Julia sets of $\frac{3z^3}{3-z^3}$ and $\frac{z^3}{1-2z^2}$}
\label{JS-form2}
\end{figure}
\subsection{Two new classes of rational maps}
A rational map $R$ is said to be exceptional if $\mathcal{J}(R)$ is the whole extended complex plane $\widehat{\mathbb{C}}$, a circle  or a line segment in $\widehat{\mathbb{C}}$. This definition is upto conformal conjugacy. In other words, each rational map which is conformally conjugate to an exceptional rational map is also exceptional and it can have an arc of a circle or a straight line as its Julia set. Examples of exceptional rational maps are well-known, e.g., the Julia set of  $\frac{z^2-2}{z^2}$ and $\frac{z^2-2}{2z}$ are $\widehat{\mathbb{C}}$ and $\mathbb{R} \cup \{\infty\}$ respectively (see Remark~\ref{exceptionalrationals}). The Julia set of every Chebyshev polynomial is $[-1,1]$ (page 11, \cite{Beardon_book}).


A  M\"{o}bius map $\phi$ is said to be of order $\beta$ if $\phi^{\beta}$, $\beta$ times composition of $\phi$ with itself, is the identity map. An irrational rotation (i.e., rotation by an angle of an irrational multiple of $2 \pi$) and a translation are  with infinite order whereas a rational rotation (i.e., rotation by an angle of a rational multiple of $2 \pi$) is of finite order. Note that $\widehat{\mathbb{C}}$ is preserved by every  M\"{o}bius map and straight lines are preserved by translations whereas  circles are preserved by irrational rotations.  These facts probably motivate a conjecture mentioned in  ~\cite{Boyd2000} which states that if there is a M\"{o}bius map of infinite order preserving the Julia set of $R$ then $R$ is exceptional. In 2000, 
 Boyd settles it partially by proving the following.
\begin{theorem}[\cite{Boyd2000}]
Let $R$ be a rational map of degree at least two such that 
$\mathcal{J}(R)=T(\mathcal{J}(R))$, where $T(z)=z+1$
and the point at infinity is either periodic or preperiodic.
Then $\mathcal{J}(R)$ is either the whole extended complex plane or a horizontal line.
\label{Boyd2000}
\end{theorem} 
The translation by $1$ in the above theorem is not any loss of generality. It is not known, at least to the present authors whether there is a non-exceptional rational map whose Julia set is not a line but is invariant under some translation.  
\par
 The fact that a translation fixes $\infty$ and $\infty$ is a periodic or pre-periodic point of $R$ are important in the above theorem. Levin generalized this result in 2001 (see \cite{Levin_LtoE2001}).
\begin{theorem}
Let $R$ be a rational map and $\phi$ be a  M\"{o}bius map  such that $\phi(\mathcal{J}(R))= \mathcal{J}(R)$. If a fixed point of $\phi$ is a periodic or pre-periodic point of $R$ and $R$ is not exceptional then $\phi$ is of finite order.
\end{theorem}
The above theorem ensures the existence of rotational symmetries for non-exceptional rational maps satisfying some conditions.

Ferreira \cite{Ferreira2019} considers all chordal isometries $z\mapsto \frac{\alpha z - \bar{\beta}}{\beta z+\bar{\alpha}}$, $|\alpha|^2+|\beta|^2=1$, that preserve rational Julia sets. Though this seems to be a natural way to generalize, two significant classes of maps, namely rotations about any non-zero point and translations, being non chordal isometries are left out of the consideration. However, all rotations about the origin are chordal isometries and some results obtained by Ferreira remains useful in our consideration of rotational symmetries. Recall that $\mathcal{I}(R)=\{s(z)=\frac{az-\bar{b}}{bz+\bar{a}}: |a|^2+|b|^2=1 \text{ and } s(\mathcal{J}(R))=\mathcal{J}(R)\}$.

%
\begin{lemma}(\cite{Ferreira2019})\label{F_prop1}
	Let $R$ be non-exceptional. For a natural number $n$ and a chordal isometry $\sigma$, if $R \circ \sigma=\sigma^n \circ R$ then $\sigma\in \mathcal{I}(R).$
\end{lemma}
Since every rotation about the origin is a chordal isometry as well as an Euclidean isometry, we have the following useful remark.
\begin{Remark}\label{Rot_nes}
 If $\sigma$ is a rotation about the origin and $R \circ \sigma=\sigma^n \circ R$ for some natural number $n$ then $\sigma \in \Sigma R$. 
\end{Remark}

The next result offers a necessary condition for a chordal isometry to be in $\mathcal{I}(R)$. 
\begin{lemma}(\cite{Ferreira2019})\label{F_prop2}
	Let $R$ be a non-exceptional rational map without any parabolic or rotation domain. If   $\sigma\in \mathcal{I}(R)$   fixes a superattracting fixed point $z_0$ of $R$ then $R \circ \sigma=\sigma^m \circ R$ where $m\geq 2$ is the local degree of $R$ at $z_0.$
\end{lemma}
Note that, if the point at infinity is not a point in the Julia set of $R$, then $\mathcal{J}(R)$ is not translation invariant. As the composition of rotations about two different points is a translation, $\Sigma R$ contains rotations about a single point whenever $\mathcal{J}(R)$ is not invariant under translation. The following result provides a class of non-exceptional rational maps $R$ such that $\Sigma R$ contains rotations about the origin.
\begin{theorem}\label{Form1}
	Let $P$ and $Q$ be two non-monomial and centered polynomials without any common factor except possibly $0$ such that   $R(z)=\frac{P(z)}{Q(z)}$ is non-exceptional rational map without any parabolic domain or a rotation domain, and its Julia set is not invariant under any (non-trivial) translation. Further, let  $P(z)=a_1 z^{\alpha_1}P_0(z^{\beta_1})$ and $Q(z)=a_2 z^{\alpha_2}Q_0(z^{\beta_2})$, where $\alpha_i,\beta_i$ are maximal for the respective expressions (refer Equation (\ref{norm})) and $a_1, a_2 \in \mathbb{C} \setminus \{0\}$ are the leading coefficients of $P$ and $Q$ respectively.  If $\alpha_1 >\alpha_2+1$  and $\beta= \gcd(\beta_1, \beta_2)>1$ then  $\Sigma R=\{z\mapsto \lambda z: \lambda^\beta=1\}$. 
\end{theorem}

\begin{proof}
It follows from the assumption that
	\begin{equation}\label{Form R_1}
	R(z)=a z^m\frac{P_0(z^{\beta_1})}{Q_0(z^{\beta_2})},
	\end{equation}
	where $a= \frac{a_1}{a_2}$ and $m=\alpha_1-\alpha_2$. Recall from (\ref{form p_0}) that $P_0 (z^{\beta_1})$ and $Q_0(z^\beta_2)$ are normalized polynomials  with non-zero constant terms. Also, $\beta=\gcd(\beta_1, \beta_2)>1$.
	\par 
	 For $\sigma(z)=\lambda z$ with $\lambda^\beta=1$, $R \circ  \sigma (z)=a\lambda^m z^m\frac{P_0(z^{\beta_1)}}{Q_0(z^{\beta_2})}=\sigma^m \circ R(z).$ This gives that $\sigma\in \Sigma R$ by the Remark \ref{Rot_nes}. Therefore, $\{z\mapsto \lambda z: \lambda^\beta=1\}\subseteq \Sigma R.$
	\par
In order to prove  $  \Sigma R \subseteq \{z\mapsto \lambda z: \lambda^\beta=1\} $, let $P_0(z)=z^{m_1}+a_2z^{m_2}+\dots +a_kz^{m_k}+a_{k+1}$ and $Q_0(z)=z^{n_1}+b_2z^{n_2}+\dots +b_rz^{n_r}+b_{r+1}$ where $a_i\neq 0$ for $i=2,3,\dots ,k+1$,  $\gcd\{m_1,m_2,\dots, m_{k}\}=1$, and $b_i\neq 0$ for $i=2,3,\dots ,r+1$ and $\gcd\{n_1,n_2,\dots, n_{r}\}=1$. In particular, $P_0$ and $Q_0$ have non-zero constant terms. All these follow from the analysis preceeding Equation~(\ref{norm}). This along with the assumption $m \geq 2$ give that $0$ is a superattracting fixed point of $R$ and $\deg(R,0)=m$.
\par  
 If there is a rotation about a non-zero point  in $\Sigma R$ then its composition with $z \mapsto \lambda z, \lambda^\beta =1$ would be a non-trivial translation and it has to be in  $\Sigma R$. But the Julia set of $R$ is not invariant under any translation by the assumption. Therefore, every element of $\Sigma R$ is a rotaion about a point depending only on $R$ and that point must be the origin because $z \mapsto \lambda z \in \Sigma R $ and $ \lambda^\beta =1$. 
 \par  Let $\sigma\in \Sigma R$ and $\sigma(z)=\mu z$ for some $|\mu|=1$.  As  $0$ is a superattracting fixed point of $R$ and $\deg(R,0)=m$, it follows from Lemma \ref{F_prop2}  that $R \circ \sigma =\sigma^m\circ R$. This implies that $$\frac{P_0(\mu^{\beta_1}z^{\beta_1})}{Q_0(\mu^{\beta_2}z^{\beta_2})}=\frac{P_0(z^{\beta_1})}{Q_0(z^{\beta_2})}.$$ Let $R_1(z)=\frac{P_0(\mu^{\beta_1}z^{\beta_1})}{Q_0(\mu^{\beta_2}z^{\beta_2})}$ and $R_2(z)=\frac{P_0(z^{\beta_1})}{Q_0(z^{\beta_2})}.$ Then $R_1$ and $R_2$ share the same sets of roots and poles. 
 
Since $P$ is not a monomial, $P_0$ is non-constant. Let  $\xi_1,\xi_2,\dots, \xi_s$ be the distinct roots of $P_0(z)$.
	Then the roots  of $R_1$ are the solutions of $z^{\beta_1}=\frac{\xi_i}{\mu^{\beta_1}}$, for $i=1,2,\dots,s$ whereas the roots  of $R_2$ are the solutions of $z^{\beta_1}=\xi_i$, for $i=1,2,\dots,s$. If there is an $i$ such that $\xi_i=\frac{\xi_i}{\mu^{\beta_1}},$ then  $\mu^{\beta_1}=1,$ and hence $Q_0(\mu^{\beta_2}z^{\beta_2})=Q_0(z^{\beta_2})$. This gives that 
	\begin{align*}
	& \mu^{n_1\beta_2}z^{n_1\beta_2}+b_2\mu^{n_2\beta_2}z^{n_2\beta_2}+\dots +b_r\mu^{n_r\beta_2}z^{n_r\beta_2}+b_{r+1}\\
	&=z^{n_1\beta_2}+b_2z^{n_2\beta_2}+\dots +b_rz^{n_r\beta_2}+b_{r+1}.
	\end{align*}
	Comparing the coefficients, we get $\mu^{n_i\beta_2}=1$ for all $i=1,2,\dots,r.$ This gives that \\ $\left(\mu^{\beta_2}\right)^{\gcd(n_1,n_2,\dots,n_r)}=1$. As $\gcd(n_1,n_2,\dots, n_r)=1,$ we get $\mu^{\beta_2}=1.$ Therefore, $\mu^{\beta}=1$ where $\beta=\gcd(\beta_1,\beta_2).$
%
	
	Now consider the remaining cases, i.e., when there is no  $i$ for which $\xi_i=\frac{\xi_i}{\mu^{\beta_1}}$. Then $\mu^{\beta_1}\neq 1$.
	\par  After renaming the roots of $R_1$ and $R_2$, if required, we can write 
	\begin{equation}\label{c1}
	\xi_1=\frac{\xi_{2}}{\mu^{\beta_1}},~ \xi_2=\frac{\xi_{3}}{\mu^{\beta_1}},~\dots,~\xi_{s-1}=\frac{\xi_{s}}{\mu^{\beta_1}} ,~\xi_s=\frac{\xi_{1}}{\mu^{\beta_1}}.
	\end{equation}
	These give that all the roots of $P_0$ are of the same modulus and any two nearest such pair of roots differ by an argument of $\frac{2 \pi}{s}$. Let $a_i$ be the multiplicity of $\xi_i$ as a root of $P_0$ for $i=1,2,3,\cdots,s$. Since $$P_0(\mu^{\beta_1} z^{{\beta_1}})=\mu^{\beta_1 (a_1+ a_2+\cdots+ a_s)}(z^{\beta_1} -\frac{\xi_1}{\mu^{\beta_1}})^{a_1}(z^{\beta_1} -\frac{\xi_2}{\mu^{\beta_1}})^{a_2}(z^{\beta_1} -\frac{\xi_3}{\mu^{\beta_1}})^{a_3}\cdots (z^{\beta_1} -\frac{\xi_k}{\mu^{\beta_1}})^{a_s},$$
	$$P_0( z^{{\beta_1}})=(z^{\beta_1} -\xi_1)^{a_1} (z^{\beta_1} -\xi_2)^{a_2}(z^{\beta_1} -\xi_3)^{a_3}\cdots (z^{\beta_1} -\xi_s)^{a_s},$$
	and
	 $P_0 (z^{\beta_1})=P_0(\mu^{\beta_1} z^{{\beta_1}})$, it follows from Equation(\ref{c1}) that $a_1=a_2=\cdots =a_s$.
	In other words, the multiplicity of each $\xi_i$ is the same. Let it be $r$. Since $\xi_i ^s$ is the same for each $i$, let it be denoted by $\xi$. Hence  $P_0(z)=(z^s-\xi)^r$. As $\beta_1$ is maximal for the expression of $P(z)=z^{\alpha_1} P_0(z^{\beta_1})$, the gcd of all the powers of $z$ in the expression of $P_0 (z)$ is $1$. However, the powers of $z$ in the expression $P_0(z)=(z^s-\xi)^r$ have gcd equal to $s$. Therefore $s=1$. It follows from Equation \ref{c1} that $\mu^{s\beta_1}=1$.  This leads to $\mu^{\beta_1}=1$ and this is a contradiction.
\end{proof}
\begin{figure}[h!]
	\begin{subfigure}{.5\textwidth}
		\centering
		\includegraphics[width=0.955\linewidth]{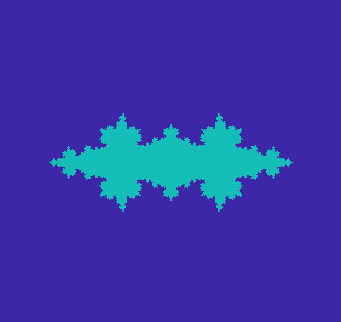}
		\caption{Julia set of $R(z)=\frac{z^2(z^2-2)}{z^2+1}$}
	\end{subfigure}%
	\begin{subfigure}{.5\textwidth}
		\centering
		\includegraphics[width=0.94\linewidth]{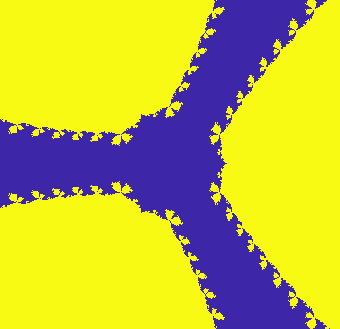}
		\caption{Julia set of $R(z)=\frac{z^3(z^3+1)}{z^6+1}$}
	\end{subfigure}
	
	\caption{Julia set of rational maps of the form $a\frac{P(z)}{Q(z)}$}
	\label{JS-form1}
\end{figure}
 Fig.~\ref{JS-form1} illustrates Theorem~\ref{Form1} for $\frac{z^2(z^2-2)}{z^2+1}$ and $\frac{z^3(z^3+1)}{z^6+1}$.

The assumption in the above theorem that $\mathcal{J}(R)$ is not invariant under any translation can be ensured by putting mild restrictions on the degrees of $P$ and $Q$.
\begin{Remark} If $\deg(P)> \deg(Q)+1$ then $R$ has a superattracting fixed point at $\infty$. On the other hand,  if $\deg(P)< \deg(Q)$, $R(\infty)=0$ and since $0$ is a superattracting fixed point of $R$, $\infty $ is in the Fatou set of $R$. In both the cases, the Julia set is bounded and therefore,  is not invariant under any translation.
\end{Remark}
The assumption in Theorem~\ref{Form1} that both $P$ and $Q$ are not monomials can be relaxed.
\begin{Remark} If both $P$ and $ Q$  are monomials  or only $Q$ but not $P$ is a monomial then $R=\frac{P}{Q}$ is polynomial itself and the issue of  rotational symmetry is already discussed in Section 2.  
	
 If only $P$ but not $Q$ is a monomial then $R=\frac{P}{Q}$ takes the form of $\frac{a z^{m}}{Q_0 (z^{\beta_2})}$ for some $a \neq 0, m>1, \beta_2 \geq 2$ where $Q_0(z^{\beta_2})$ is a normalized polynomial in $z$ with non-zero constant term. Note that $0$ is a superattracting fixed point of $R$ and its Julia set cannot be $\widehat{\mathbb{C}}$. If the Julia set of $R$ is not invariant under any translation then $\Sigma R=\{z \mapsto \lambda z: \lambda^{\beta_2}=1\}$. The proof of  $\{z \mapsto \lambda z: \lambda^{\beta_2}=1\} \subseteq \Sigma R$ is the same as the first part of the proof of Theorem \ref{Form1}. In order to prove that $\Sigma R \subseteq \{z \mapsto \lambda z: \lambda^{\beta_2}=1\} $, first note that every element of $\Sigma R$ is a rotation about the origin (this follows from the arguments used in the proof of Theorem \ref{Form1}). Let $Q_0(z)=z^{n_1}+b_2z^{n_2}+\dots +b_rz^{n_r}+b_{r+1}$ where  $b_i\neq 0$ for $i=2,3,\dots ,r+1$ and $\gcd\{n_1,n_2,\dots, n_{r}\}=1$. If $z \mapsto \mu z \in \Sigma R$ then $R(\mu z)=\mu ^m R(z)$ by Lemma \ref{F_prop2}. Consequently,  $Q_0(\mu^{\beta_2}z^{\beta_2})=Q_0(z^{\beta_2})$. In other words, 
 \begin{align*}
 & \mu^{n_1\beta_2}z^{n_1\beta_2}+b_2\mu^{n_2\beta_2}z^{n_2\beta_2}+\dots +b_r\mu^{n_r\beta_2}z^{n_r\beta_2}+b_{r+1}\\
 &=z^{n_1\beta_2}+b_2z^{n_2\beta_2}+\dots +b_rz^{n_r\beta_2}+b_{r+1}.
 \end{align*}
 Comparing the coefficients, we get $\mu^{n_i\beta_2}=1$ for all $i=1,2,\dots,r.$ This gives that \\ $\left(\mu^{\beta_2}\right)^{\gcd(n_1,n_2,\dots,n_r)}=1$. Since $\gcd(n_1,n_2,\dots, n_r)=1,$  $\mu^{\beta_2}=1.$
 
 \end{Remark}
Using Theorem~\ref{Form1}, a partial generalization of Theorem~\ref{poly-symm} is possible.
\begin{theorem}\label{Form2}
	Let $P$ be a normalized polynomial of degree $d$ and $P(z)=z^\alpha P_0(z^\beta)$ where $\alpha,\beta$ are maximal for this expression (refer Equation (\ref{norm})).  Also, let  $R(z)=az^\nu P(z)$  where $a \in \mathbb{C} \setminus \{0\}$ and $\nu \in \mathbb{Z}$ such that it has no parabolic domain or any rotation domain. If $ \beta \geq 2$  then $\Sigma P = \Sigma R$  for all $\nu$ except the possible values of $ -d$ and $  -d+1$. For $\nu= -d$ or $-d+1$, we have $\Sigma P \subseteq \Sigma R$.
	
\end{theorem}
\begin{proof}
	Since   $\beta \geq 2$, it follows from Theorem~\ref{poly-symm} that  $\Sigma P=\{z\mapsto \lambda z: \lambda^\beta=1\}$ contains a non-identity map. Note that $$R(z)=\frac{a P_0(z^\beta)}{z^{-(\nu+\alpha)}}.$$

For $\nu\geq -\alpha$, $R$  is itself a centered polynomial and  we have   $\Sigma P =\Sigma R$ by the remark following Theorem~\ref{poly-symm}.
	 
Let $\nu<-\alpha$. Choose a positive integer $m$ such that $m-(\alpha+\nu)$ is a positive multiple of $\beta$. Then for every  $\sigma (z)=\lambda z$ with $\lambda^\beta=1$, $R(\sigma(z))=\lambda^m R(z)$. Hence, by Lemma~\ref{F_prop1}, $\Sigma P\subseteq \Sigma R$. Note that if $\mathcal{J}(R)$ is not $\widehat{\mathbb{C}}$ and  is not invariant under any translation then each element of $\Sigma R$ is a rotation about the origin. This is to be used later.
\par 

Now we look into the possibility of $\Sigma R \subseteq \Sigma P $ and hence the equality of  $\Sigma P $ and $ \Sigma R$ when $\nu<-\alpha$. 

Let $-d+2\leq\nu<-\alpha$. 
	
\par  If  $-d+2=-\alpha$ (this happens when $\deg (P_0)=1$ and $\beta =2$) then  $\nu=-\alpha$. It is already found that $\Sigma P\subseteq \Sigma R$ in this case. We proceed with the other situation, i.e.,  $-d+2 <-\alpha$.  Since  $-d+2\leq\nu$,  $\alpha -d+2\leq \alpha +\nu$. Let $\deg(P_0)=m_1$. Then  $m_1 \beta =d -\alpha  \geq 2-(\alpha+\nu)$. This gives that
$\infty$ is a superattracting fixed point of $R$ with local degree $d+\nu$. Then, for any $\sigma\in \Sigma R$, where $\sigma(z)=\lambda z$, $R(\sigma(z))=\sigma^{d+\nu}(R(z))$ by Lemma~\ref{F_prop2}. This gives that $P_0(\lambda^\beta z^\beta)=\lambda^{m_1 \beta } P_0(z^\beta)$. Let $P_0 (z)=z^{m_1}+a_2 z^{m_2}+\cdots +a_{k}z^{m_k}+a_{k+1}$ where each $a_i \neq 0$ and $\gcd(m_1, m_2,\cdots m_k )=1$. Comparing the coefficients except the leading term, we get $\lambda^{m_i \beta}=1$ for all $i=1,2,3,\cdots, k$. Since $\gcd \{m_1,m_2,m_3,\cdots,m_k\}=1$ we have $\lambda^{\beta}=1$. Therefore, $\Sigma R\subseteq \{z\mapsto \lambda z: \lambda^\beta=1\}$. It is important to note here that there is no translation in $\Sigma R$ as $\infty$ is in the Fatou set of $R$.

\par 
Now consider  $\nu<-d$.

 Let $\zeta=-(\nu+\alpha)$. Then $R(z)=a\frac{P_0(z^\beta)}{z^\zeta}$. As $\zeta> d-\alpha=m_1\beta$, $R(\infty)=0$ and $R(0)=\infty$. Now, $\deg(R,0)=\zeta \geq 2$ gives that $\{0,\infty\}$ is a $2$-cycle of superattracting periodic points of $R$.
 For $F(z)=R(R(z))$, each of $ 0$ and $ \infty$ is a superattracting fixed point of $F$.

  Let $\xi_1,\xi_2,\dots, \xi_s$ be the distinct roots of $P_0(z)$, i.e., $P_0(z)=(z-\xi_1)^{a_1}(z-\xi_2)^{a_2}\dots(z-\xi_s)^{a_s}$, where $\sum_{i=1}^{s}a_i=m_1$. 
Then
\begin{align*}
F(z)&=R\left(a\frac{P_0(z^\beta)}{z^\zeta}\right)=a\frac{P_0\left(\left(a\frac{P_0(z^\beta)}{z^\zeta}\right)^\beta\right)}{\left(a\frac{P_0(z^\beta)}{z^\zeta}\right)^\zeta}\\
&=\frac{z^{\zeta^2}P_0\left(\frac{a^\beta(P_0(z^\beta))^\beta}{z^{\beta\zeta}}\right)}{a^{\zeta-1}(P_0(z^\beta))^\zeta}.
\end{align*}
Now,  \begin{align*}
&P_0\left(\frac{a^\beta(P_0(z^\beta))^\beta}{z^{\beta\zeta}}\right)\\&=\left(\frac{a^\beta}{z^{\beta\zeta}}(P_0(z^\beta))^\beta-\xi_1\right)^{a_1}\left(\frac{a^\beta}{z^{\beta\zeta}}(P_0(z^\beta))^\beta-\xi_2\right)^{a_2}\dots \left(\frac{a^\beta}{z^{\beta\zeta}}(P_0(z^\beta))^\beta-\xi_s\right)^{a_s}\\
&=\frac{1}{z^{m_1\beta\zeta}}\left(a^\beta(P_0(z^\beta))^\beta-\xi_1z^{\beta\zeta}\right)^{a_1}\left(a^\beta(P_0(z^\beta))^\beta-\xi_2z^{\beta\zeta}\right)^{a_2}\dots \left(a^\beta(P_0(z^\beta))^\beta-\xi_sz^{\beta\zeta}\right)^{a_s}.
\end{align*}
The numerator of this expression is a polynomial,  each of whose non-constant terms is of the  form $z^{l \beta }$ for some $l>0$. Since $\beta \geq 2$, the difference of two consecutive powers is at least $2$. In particular,  this polynomial is centered. Let $A P_1(z^{\beta_1})$ be this polynomial where $P_1(z^{\beta_1})$ is a normalized polynomial in $z$  and such that $\beta_1$ is maximal for this expression. Clearly $\beta$ divides $\beta_1$. By the same argument,
\begin{align*}
(P_0(z^\beta))^\zeta&=(z^\beta-\xi_1)^{\zeta a_1}(z^\beta-\xi_2)^{\zeta a_2}\dots (z^\beta-\xi_s)^{\zeta a_s}
\end{align*}
is a normalized polynomial. This expression can be written as $P_2(z^\beta)$ for some polynomial $P_2$ where $\beta$ is maximal for this expression as it is so for $P_0(z^\beta)$. Hence, we get $$F(z)=Bz^{\zeta^2-m_1\beta\zeta}\frac{P_1(z^{\beta_1})}{P_2(z^\beta)}.$$
where $B=\frac{A}{a^{\zeta-1}}$.  Note that    $\zeta^2-m_1\beta\zeta=\zeta(\zeta-m_1\beta)>2$ and $\gcd(\beta_1,\beta)=\beta\geq 2$. As $0$ and $\infty$ are superattracting fixed points of $F$, these are in $\mathcal{F}(F)$. Hence $\mathcal{J}(F)$ is not invariant under any translation. Also, there is no common root of $P_1 (z^{\beta_1})$ and $P_2 (z^\beta)$. It follows from  Theorem \ref{Form1} that $\Sigma F= \{z\mapsto \lambda z: \lambda^\beta=1\}$. As $\mathcal{J}(R)=\mathcal{J}(F)$, we conclude that $\Sigma R = \{z\mapsto \lambda z: \lambda^\beta=1\}$.

\end{proof}
\begin{Remark}\label{exc}
	The Julia set of the  rational map $R$ as given in the above theorem  cannot be $\widehat{\mathbb{C}}$ or a line.   For $\nu = -d$ or $-d+1$, it is possible.
	\end{Remark}
	
\begin{Remark}\label{exceptionalrationals}
Let $ \nu=-d$. The Julia set of $R_1(z)=\frac{z^2-2}{z^2}$  is $\widehat{\mathbb{C}}$ and $\Sigma R_1$ consists of all the Euclidean isometries. Here $\nu =-2$, $d=2$ and therefore $\nu=-d$. Writing $R_1(z)$ as $\frac{P(z)}{z^2}$, we see that $\Sigma P=\{\lambda z: \lambda^2 =1\}$ and 
	$\Sigma P \subsetneq \Sigma R_1$. 
	
Consider $\tilde{R}(z)= \frac{z^2 -1}{z^2}$. As per the notations of Theorem~\ref{Form2}, $\nu=-d=-2$ and $P(z)=z^2 -1$. Note that $\deg(\tilde{R})=2$, so $\tilde{R}$ has two critical points, namely $0$ and $\infty$. Also, $\{0, \infty, 1\}$ is a cycle of $3$-periodic points of $\tilde{R}$ containing both the critical points. Therefore this is a superattracting cycle and hence, $0, \infty \in \mathcal{F}(\tilde{R})$. This leads to the conclusion that Julia set of $\tilde{R}$ is not invariant under any translation and $\tilde{R}$ does not have any parabolic or rotation domain. Let $R=\tilde{R} \circ \tilde{R} \circ \tilde{R}$. Then $R(z)=-z^4 \frac{z^4 -4 z^2+2}{(2z^2 -1)^2}$. Since $\mathcal{J}(\tilde{R})=\mathcal{J}( R)$, the latter is not invariant under any translation. Also all other assumptions of Theorem~\ref{Form1} are satisfied and it follows that $\Sigma R=\{z \mapsto \lambda z: \lambda^2=1\}$. Hence $\Sigma \tilde{R}=\{z \mapsto \lambda z: \lambda^2=1\}$ which is nothing but $\Sigma P$. The Julia set of $\tilde{R}$ is given in red in Fig.~(\ref{three-periodic-basin}).
	\end{Remark}
\begin{figure}[h!]
	\centering
	\includegraphics[width=0.4\linewidth]{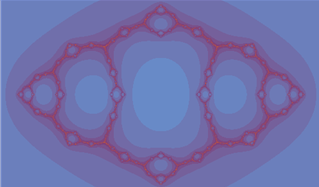}
	\caption{Julia set of $ \frac{z^2-1}{z^2} $}
	\label{three-periodic-basin}
\end{figure}
	 \begin{Remark} To discuss the situation $\nu=-d+1$, consider the Newton map $N_1$ applied to $z^2+1$. Then $N_1(z)=\frac{z^2-1}{2z}$, and $\mathcal{J}(N_1)=\mathbb{R} \cup \{\infty\}$. Thus $\Sigma N_1$ contains translations by every real number and the rotation about the origin by an angle of $\pi$. However, $N_1$ can be written as $N_1(z)=\frac{P(z)}{2z}$ where $P(z)=z^2 -1$. In this case $\nu =1$ and $d=2$, therefore, $\nu=-d+1$. Note that $\Sigma P =\{\lambda z: \lambda^2 =1\}$. Thus $\Sigma P \subsetneq \Sigma N_1$. Here $N_1$ is exceptional.
	 \par
	 Now consider the Newton map $N_2$ applied to the polynomial $z^3-1$. Then $N_2(z)=\frac{2z^3+1}{3z^2}$. The Fatou set $\mathcal{F}(N_2)$ consists of the basins of the superattracting fixed points of $N_2$ corresponding to the roots of $z^3-1=0$, and the Julia set $\mathcal{J}(N_2)$ is connected. Therefore, $N_2$ is non-exceptional rational map without any parabolic domain or a rotation domain, and its Julia set is not invariant under any (non-trivial) translation. Also note that, $\mathcal{F}(N_2)$ contains exactly three unbounded components, namely the immediate basins of the superattracting fixed points of $N_2$. Express $N_2$ as $N_2(z)=\frac{2 P(z)}{3 z^2}$ where $P(z)= z^3 +\frac{1}{2}$, $d=3$ and  $\nu=-2$. As  $\{z\mapsto \lambda z:\lambda^3=1\}\subseteq \Sigma N_2$ (from the first part of the Theorem \ref{Form2}), the elements of $\Sigma N_2$ are rotations about the origin and are of finite order. If there is a rotation $\sigma$ of order bigger than three preserving $\mathcal{J}(N_2)$, then an unbounded component of $\mathcal{F}(N_2)$ will be mapped to a bounded component of $\mathcal{F}(N_2)$ by $\sigma$, which is not possible. Since $\Sigma P =\{z\mapsto \lambda z:\lambda^3=1\}$, $\Sigma P= \Sigma N_2$.
\end{Remark}
\begin{figure}[h!]
	\centering
	\includegraphics[width=0.4\linewidth]{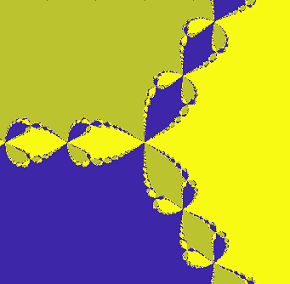}
	\caption{Julia set of $ N_2(z)=\frac{2z^3+1}{3z^2} $}
	\label{Eq_sym}
\end{figure}
We conclude with a corollary, which is already known and can be found in \cite{Ferreira2019}. It is presented here with a simplified proof using Theorem \ref{Form2}.
   \begin{corollary}
   	 For the McMullen map $R_\lambda(z)= z^m+\frac{\lambda}{z^n}$, where $m,n \in \mathbb{N}$ with $m\geq 2$ and $\lambda \in \mathbb{C}\setminus \{0\}$, if $R_\lambda(z)$ has no parabolic or rotation domain then $\Sigma R_\lambda=\{\sigma: \sigma(z)=\mu z, \mu^{m+n}=1\}$.
   \end{corollary}
\begin{proof}
	Here $R_\lambda(z)= z^m+\frac{\lambda}{z^n}=\frac{z^{m+n}+\lambda}{z^n}=z^\nu P(z)$, where $\nu=-n$ and $P(z)=z^{m+n}+\lambda$, which is a normalized polynomial with $\Sigma P=\{z\mapsto \mu z: \mu^{m+n}=1\}$. Also, note that $-(m+n)+2\leq\nu<0$. By Theorem \ref{Form2}, $\Sigma R_\lambda=\Sigma P.$
\end{proof}

	\section{Newton's and Chebyshev's methods}
A root-finding method applied to a non-constant polynomial $p$ is a rational map, for which every root of $p$ is an attracting (superattracting if the root is simple) fixed point. A root-finding method can be expected to inherit, at least partially, some dynamical aspects of the polynomial. This section deals with the possible relation between the rotational symmetries of the Julia set of a  polynomial and that of some root-finding methods applied to it.

\par  The Newton's method is a classical root-finding method. The Newton's method $N_p$ applied to a polynomial $p$ is defined as $$N_p(z)=z-\frac{p(z)}{p'(z)}.$$
The dynamics of the Newton's method applied to a polynomial has already been determined. It is proved that the Julia set of $N_p$ is connected. It is due to a result of Shishiura (Theorem \cite{Shishikura}) who proved that the Julia set of a rational map of degree at least two is disconnected if the rational map has at least two weakly repelling fixed point (a fixed point which is either repelling or a multiple fixed point). Other than the roots of $p$ (which are indeed attracting fixed points of $N_p$), $\infty$ is only a fixed point of $N_p$ and it is repelling. The following is proved by Yang \cite{Yang2010}.
\begin{theorem}\label{Yang}
	 If $p$ is a normalized polynomial  then $\Sigma p\subseteq \Sigma N_p$.
\end{theorem}
In the same article (\cite{Yang2010}) it is proved that the Julia set of $N_p$ is a line whenever $p$ has exactly two roots with same multiplicity. In fact, the Julia set is the perpendicular bisector of the line segment joining two roots of $p$. In this case $\Sigma N_p$ contains translation, hence can not be equal to $\Sigma p$. If $\Sigma N_p$ does not contain any translation and $\Sigma p$ is non-trivial, then Theorem \ref{Yang} asserts that $\Sigma N_p$ contains rotations about the origin. In this scenario one can expect equality in $\Sigma p$ and $\Sigma N_p$. We apply certain conditions on $p$ and use Theorem \ref{Form1} to prove the desire equality. As $\mathcal{J}(N_p)$ is connected, $N_p$ does not contain a Herman ring of any period. However, the possibility of existence of a Siegel disk can not be discarded. Note that, a polynomial is called generic if all its roots are simple.
\begin{theorem}\label{New_equality}
	Let $p$ be a normalized, generic polynomial $p$ with a root at the origin and $\Sigma p$ be non-trivial. If $N_p$ does not contain a parabolic domain or a Siegel disk then $\Sigma p=\Sigma N_p$.
\end{theorem}
\begin{proof}
	Recall from  (\ref{norm}) and (\ref{form p_0}) that   $p(z)=z^\alpha p_0(z^\beta)$, where $\alpha, \beta$ are maximal for this expression and $p_0$ is a monic polynomial. Since $0$ is a simple root of $p$, $\alpha=1$ and as $\Sigma p$ contains more than one element, $\beta\geq 2$. The latter gives that every root of $p_0$ gives rise to two distinct roots of $p$. Therefore, $p$ has at least three distinct roots and hence the Fatou set of $N_p$ contains at least three (in fact, infinitely many) Fatou components. The Julia set cannot be a line.  Since $\infty$ is a repelling fixed point of $N_p$, it follows from Theorem~\ref{Boyd2000} that $\mathcal{J}(N_p)$ is not invariant under any translation. Further, $N_p$ is not exceptional.
	\par  Now, $p$ can be writtes as
	$$p(z)=z\sum_{i=1}^{k}a_iz^{m_i\beta}+a_{k+1}z$$ where  $p_0(z)=z^{m_1}+a_2z^{m_2}+\dots +a_kz^{m_k}+a_{k+1}$,  $a_1=1$ and $a_i \neq 0$ for $i=2,3,\cdots k+1$. Then $p'(z)=\sum_{i=1}^{k}b_iz^{m_i\beta}+a_{k+1}$ where $b_i=a_i(m_i\beta+1)$, and
	\begin{align*}
	N_p(z)&=z-\frac{p(z)}{p'(z)}=z-\frac{z (\sum_{i=1}^{k}a_iz^{m_i\beta})+a_{k+1}z}{(\sum_{i=1}^{k}b_iz^{m_i\beta})+a_{k+1}}\\
	&=\frac{z\left(m_1\beta z^{m_1\beta}+a_2m_2\beta z^{m_2\beta}+\dots+a_km_k\beta z^{m_k\beta}\right)}{(\sum_{i=1}^{k}b_iz^{m_i\beta})+a_{k+1}}.
	\end{align*}

	Let $$P(z)=\frac{1}{m_1 \beta} z\left(m_1\beta z^{m_1\beta}+a_2m_2\beta z^{m_2\beta}+\dots a_{k-1} m_{k-1} \beta z^{m_{k-1}\beta}+a_km_k\beta z^{m_k\beta}\right)$$ and $$Q(z)=\frac{1}{b_1}\big[ (\sum_{i=1}^{k}b_iz^{m_i\beta})+a_{k+1}\big]. $$
	Then $P $ and $Q$ are normalized polynomials without any common root (as $p$ is generic) and     $N_p (z)=\frac{m_1 \beta}{b_1}\frac{P(z)}{Q(z)}$. Let $P_0(z)=z^{m_1-m_k}+\frac{a_2 m_2}{m_1}z^{m_2 -m_k}+\cdots+ \frac{a_{k-1} m_{k-1}}{m_1}z^{m_{k-1} -m_k}+\frac{a_k m_k}{m_1}$. Then  $P(z)=z^{m_k \beta +1} P_0(z^{\beta_1})$ where $m_k \beta +1$ and $\beta_1$ are maximal for this expression. Here $\beta_1=\beta \gcd(m_1 -m_k, m_2 -m_k,\cdots, m_{k-1}-m_k)$ is a multiple of $\beta$. 
	Similarly $Q(z)=Q_0(z^\beta)$ where $\beta$ is maximal for this expression. Now, 
	
	$$
	N_p (z)=\frac{m_1 \beta}{b_1} \frac{z^{m_k \beta +1} P_0 (z^{\beta_1})}{Q_0(z^\beta)}. $$
As $\gcd(\beta_1,\beta)=\beta \geq 2$ and $m_k \beta +1 >1$, it follows from Theorem [\ref{Form1}] that $\Sigma N_p=\{z\mapsto\lambda z:\lambda^\beta=1\}$, which is nothing but $\Sigma p.$
\end{proof}
The polynomial $p(z)=z(z^3-1)$ is normalized, generic and is with a simple root at the origin. Also $\Sigma p=\{z\mapsto \lambda z: \lambda^3=1\}$, non-trivial. Here $N_{p}(z)=\frac{3z^4}{4z^3-1}$. Note that the critical points of $N_{p}$ are the roots of $p$, which are also superattracting fixed points of $N_{p}$. Hence $N_{p}$ can not have any parabolic domain or Siegel disk. Thus, it follows from Theorem~\ref{New_equality} that $\Sigma p=\Sigma N_{p}$ (see Fig. \ref{New_sym}(a)).
\par 
The preceding example can be generalized with the same argument. Consider the polynomial $p(z)=z(z^n-1),~n\geq 2$. Then $N_p(z)=\frac{nz^{n+1}}{(n+1)z^n-1}$. The critical points of $N_p$ are the roots of $p$, and the Fatou set of $N_p$ consists of the basin of attractions corresponding to the roots of $p$. Therefore, $N_p$ does not have any parabolic or rotation domain. Hence by Theorem \ref{New_equality}, $\Sigma p=\Sigma N_p=\{z\mapsto \lambda z: \lambda^n=1\}$.
\begin{figure}[h!]
	\begin{subfigure}{.5\textwidth}
		\centering
		\includegraphics[width=0.98\linewidth]{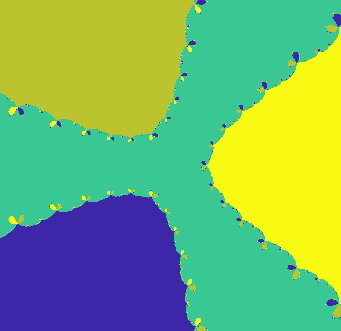}
		\caption{Julia set of $ N_{p}(z)$}
	\end{subfigure}%
	\begin{subfigure}{.5\textwidth}
		\centering
		\includegraphics[width=0.975\linewidth]{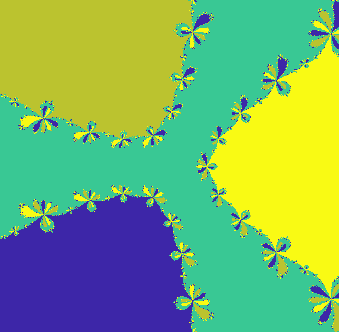}
		\caption{Julia set of $ C_{p}(z)$}
	\end{subfigure}
	\caption{The Newton's method and the Chebyshev's method applied to $p (z)=z(z^3-1)$}
	\label{New_sym}
\end{figure}
 \begin{Remark}
 The K\"{o}nig's methods is a family of root-finding methods $\{K_{p,n}, n=2,3,\dots \}$  defined by $$K_{p,n}(z)=z+(n-1)\frac{\left(\frac{1}{p}\right)^{[n-2]}(z)}{\left(\frac{1}{p}\right)^{[n-1]}(z)},$$ where $\left(\frac{1}{p}\right)^{[i]}$ denotes the $i$-th derivative of $\frac{1}{p}$. For $n=2$, $K_{p,2}$ is the Newton's method.  Some useful properties of $K_{p,n}$ can be found in \cite{BuffHenriksen2003}. Liu and Gao proved the following results in ~\cite{Liu_Gao2015}.
\begin{theorem}
If $p$ is a normalized polynomial then $\Sigma p\subseteq \Sigma K_{p,n}$. 
\label{konig-symmetry}
\end{theorem} 
\begin{theorem}
	The  Julia set $\mathcal{J}(K_{p,n})$ of $K_{p,n}$ is a straight line if and only if $p$ has exactly two distinct roots with the same multiplicity.
\end{theorem} 
Since $\infty$ is a repelling fixed point of $K_{p,n}$ and $\mathcal{J}(K_{p,n}) \neq \widehat{\mathbb{C}}$, it follows from Theorem~\ref{Boyd2000} that $\mathcal{J}(K_{p,n})$ is not invariant under any translation whenever $p$ has exactly two distinct roots with different multiplicities or has at least three distinct roots. In this case, there are rotational symmetries of $\mathcal{J}(K_{p,n})$ whenever $\Sigma p$ is non-trivial. This motivates the following conjecture.
\begin{Conjecture}
For $n \geq 2$ and a normalized polynomial 
$p$, $\Sigma K_{p,n}=\Sigma p$.	
\end{Conjecture}
It is already proved in Theorem~\ref{New_equality}
that this is true for $n=2$ under certain conditions.
\end{Remark}
We now prove Theorem~\ref{konig-symmetry} for another root-finding method, namely the  Chebyshev's method. The Chebyshev's method of $p$, denoted by $C_p$ is defined as $$C_p(z)=z-\left(1+\frac{1}{2}L_p(z)\right)\frac{p(z)}{p'(z)}$$
where $L_p(z)=\frac{p(z)p''(z)}{(p'(z))^2}$. Note that if $p$ is a monomial or a linear polynomial then $C_p$ is a linear polynomial, and this is not of interest here. Now onwards, we assume that $p$ is not a monomial and $\deg(p)\geq 2$.
The Chebyshev's method is a third-order convergent method. In other words, the local degree of $C_p$ at each simple root of $p$ is at least three.
\begin{theorem}\label{Sym_Cheb}
	For every normalized polynomial $p$, $\Sigma p\subseteq \Sigma C_p$.
\end{theorem}
\begin{proof}
	Recall  that  $p(z)=z^\alpha p_0(z^\beta)$  where  $\alpha, \beta$ are maximal for this expression. If $\beta=1$, then $\Sigma p$ contains the identity map only  and  hence the theorem is trivial. 
	\par Let $\beta\geq 2$. Then $p$ has at least two non-zero roots. This is because  $p_0(z^\beta)$ has a non-zero root and $\beta \geq 2$. For every $\sigma \in \Sigma p, \sigma(z)=\lambda z$  where $\lambda^\beta=1$. Note that $p(\lambda z)=\lambda^\alpha z^\alpha p_0(\lambda^\beta z)=\lambda^\alpha p(z)$. Differentiating it once and twice, it is found that
	$p'(\lambda z)=\lambda^{\alpha-1} p'(z)$ and $
	 p''(\lambda z)=\lambda^{\alpha-2} p''(z).$
	Thus, $L_p(\sigma(z))=\frac{\lambda^{2\alpha-2}p(z)p''(z)}{(\lambda^{\alpha-1}p'(z))^2}=L_p(z)$ and hence
	\begin{align*}
	C_p(\sigma(z))&=\sigma(z)-\left(1+\frac{1}{2}L_p(\sigma(z))\right)\frac{p(\sigma(z))}{p'(\sigma(z))}\\
	&=\lambda z-\left(1+\frac{1}{2}L_p(z)\right)\frac{\lambda p(z)}{p'(z)}=\sigma(C_p(z)).
	\end{align*}
	Therefore $\sigma(\mathcal{J}(C_p) =\mathcal{J}(C_p)$ and $\sigma\in \Sigma C_p$.
\end{proof}
A generalized version of Theorems \ref{Yang}, \ref{konig-symmetry} and \ref{Sym_Cheb} is established in \cite{Sym-and-dyn}. It is shown that if a root-finding method $F$ satisfies a special property, named as the Scaling Theorem, then for any normalized polynomial $p$, $\Sigma p\subseteq \Sigma F_p$ (Theorem 1.1., \cite{Sym-and-dyn}). Note that the K\"{o}nig's methods and the Chebyshev's method satisfy the Scaling Theorem (see \cite{BuffHenriksen2003} and \cite{Nayak_Pal}). An attempt to find equality in $\Sigma p$ and $\Sigma C_p$ is made in \cite{Sym-and-dyn} and \cite{Quartic}. The work is done for  polynomials of degree up to four. In a general case, the same is done for unicritical polynomial and polynomials with exactly two roots. Considering a polynomial of any degree, the following result proves equality in $\Sigma p$ and $\Sigma C_p$ under certain conditions.
\begin{theorem}\label{Eq_Cheby}
Let  $p$ be a normalized and generic polynomial with a simple root at the origin such that $p'$ is also generic. If $\Sigma p$ is non-trivial and $C_p$ does not contain a parabolic or rotation domain then $\Sigma p=\Sigma C_p$.
\end{theorem}
\begin{proof}
The Chebyshev's method has a repelling fixed point at $\infty$.	Using the arguments used in the proof of the Theorem \ref{New_equality}, it is found that $\mathcal{J}(C_p) \neq \widehat{\mathbb{C}}$ and $\Sigma C_p$ does not contain any translation.  For $p(z)=z\sum_{i=1}^{k}a_iz^{m_i\beta}+a_{k+1}z$ where $a_1=1$, observe that $p'(z)=\sum_{i=1}^{k}b_iz^{m_i\beta}+a_{k+1}$, where $b_i=a_i(m_i\beta+1)$, and $p''(z)=\sum_{i=1}^{k}c_iz^{m_i\beta-1}$,
	where $c_i=b_im_i\beta$, for $i=1,2,\dots, k$ (see (\ref{form p_0})). Note that $\Sigma p$ is non-trivial amounts to $\beta \geq 2$.
We have
$$
	C_p(z)=z-\left(1+\frac{1}{2}L_p(z)\right)\frac{p(z)}{p'(z)}=z\left(\frac{2(p'(z))^3-\{2(p'(z))^2-zp''(z)p_0(z^\beta)\}p_0(z^\beta)}{2(p'(z))^3}\right)
.$$
Since $p,p'$ are generic, there is no common root of $F_1(z)-F_2 (z)$ and $ p'(z))^3$.
Let  $F_1(z)=2(p'(z))^3$  and $F_2(z)=\{2(p'(z))^2-\-zp''(z) p_0(z^\beta)\}p_0(z^\beta)$. Then $$	C_p(z) =z\left(\frac{F_1(z)-F_2(z)}{2(p'(z))^3}\right).$$

 Since $F_1(z)=2\left(\sum_{i=1}^{k}b_iz^{m_i\beta}+a_{k+1}\right)^3$ and,
	\begin{align*}
	&F_2(z)= \\
	& \left[2\left(\sum_{i=1}^{k}b_iz^{m_i\beta}+a_{k+1}\right)^2-\left(\sum_{i=1}^{k}c_iz^{m_i\beta}\right)\left(\sum_{i=1}^{k}a_iz^{m_i\beta}+a_{k+1}\right)\right]\left(\sum_{i=1}^{k}a_iz^{m_i\beta}+a_{k+1}\right).
	\end{align*}
	The constant terms of $F_1(z)$ and $F_2(z)$ are the same and that is $2a_{k+1}^3$. In other words,  there is no constant term in  $F_1(z)-F_2(z)$. Further, each of its terms is a constant multiple of some positive power of  $z^\beta$.
	\par 
	
	 Hence, there exist natural numbers $\zeta$ and $\beta_1$ such that $F_1(z)-F_2(z)=Az^\zeta F(z^{\beta_1})$, where $\beta$ divides $\zeta $ as well as $\beta_1$  and $A$ is the coefficient of the leading term of $F_1 (z)-F_2 (z)$.  Choose $\zeta $ and $\beta_1$ to be maximal for this expression. Note that $\zeta$ is the multiplicity of $0$ as a root of $F_1(z)-F_2 (z)$. By the choice of $A$, $F(z^{\beta_1})$ is monic. Further, since $\beta_1 \geq 2$,    $F(z^{\beta_1})$ is centered. Thus $F(z^\beta_1)$ is a normalized polynomial and is  with a non-zero constant term because of the choice of $\zeta$.
	 	Note that $(p'(z))^3$ is centered as $\beta \geq 2$ and taking its leading coefficient, say $A'$ common we get the expression of $C_p$ as
	$$C_p(z)=B z^{\zeta+1}\frac{F(z^{\beta_1})}{Q(z^\beta)},$$ where $B=\frac{A}{2A'}$ and $(p'(z))^3 =A' Q(z^\beta) $. For this latter expression, $\beta$ is maximal as it is so for $p'$. Since $\gcd(\beta_1,\beta)=1$, it follows from Theorem~\ref{Form1} that $\Sigma C_p= \{z\mapsto\lambda z:\lambda^\beta=1\}$. Therefore $\Sigma C_p =\Sigma p$.
\end{proof}
Consider the polynomial $p(z)=z(z^3-1)$. Then $p'(z)=4z^3-1$ is generic. In \cite{Sym-and-dyn} $\Sigma p=\Sigma C_{p}$ is proved by analyzing the number of unbounded Fatou components of $C_{p}$. It is also proved that the $\mathcal{J}(C_{p})$ is connected and $\mathcal{F}$ is the union of basins of attraction of the superattracting fixed points of $C_{p}$ corresponding to the roots of $p$. Thus $\mathcal{F}(C_{p})$ does not consist of any parabolic or rotation domain. The Fig. (\ref{New_sym}(b), illustrating the Fatou and the Julia set of $C_{p}$, supports Theorem \ref{Eq_Cheby}.

\end{document}